\documentclass[coll]{imsart}

\usepackage{amsthm,amsmath,amsfonts,natbib}
\usepackage{url}
\arxiv{arXiv:1206.4762v2}

\startlocaldefs

\newcommand{\NormalDis}{\mathcal{N}}
\newcommand{\ChiSqDis}{\mathop{\rm chi}\nolimits^2}

\newcommand{\norm}[1]{\lVert #1 \rVert}
\newcommand{\abs}[1]{\lvert #1 \rvert}

\newcommand{\bignorm}[1]{\left\lVert #1 \right\rVert}
\newcommand{\bigabs}[1]{\left\lvert #1 \right\rvert}
\newcommand{\bigset}[1]{\left\{\, #1 \,\right\}}

\newcommand{\asto}{\stackrel{\text{as}}{\longrightarrow}}
\newcommand{\probto}{\stackrel{P}{\longrightarrow}}
\newcommand{\weakto}{\stackrel{\mathcal{L}}{\longrightarrow}}

\newcommand{\real}{\mathbb{R}}

\DeclareMathOperator{\var}{var}

\newcommand{\opand}{\mathbin{\rm and}}
\newcommand{\opor}{\mathbin{\rm or}}

\newtheorem{theorem}{Theorem}[section]
\newtheorem{lemma}[theorem]{Lemma}
\newtheorem{corollary}[theorem]{Corollary}

\endlocaldefs

\begin{document}

\begin{frontmatter}

% "Title of the paper"
\title{Asymptotics of Maximum Likelihood
without the LLN or CLT or Sample Size Going to Infinity}
\runtitle{Maximum Likelihood}

% indicate corresponding author with \corref{}
\author{\fnms{Charles J.} \snm{Geyer}\corref{}\ead[label=e1]{geyer@umn.edu}}
% \thankstext{t1}{Thanks to somebody} 
\address{School of Statistics University of Minnesota, Minneapolis, MN 55455 
\printead{e1}}
\affiliation{University of Minnesota}

%\author{\fnms{???} \snm{???}\ead[label=e1]{???}}
%\address{\printead{e1}}
%\and
%\author{\fnms{???} \snm{???}\ead[label=e2]{???}}
%\address{\printead{e2}}
%\affiliation{???}

\runauthor{C.~J. Geyer}

\begin{abstract}
If the log likelihood is approximately quadratic with constant Hessian,
then the maximum likelihood estimator (MLE) is approximately
normally distributed.  No other assumptions are required.
We do not need independent and identically distributed data.
We do not need the law of large numbers (LLN)
or the central limit theorem (CLT).
We do not need sample size going to infinity or
anything going to infinity.

Presented here is a combination of Le Cam style theory
involving local asymptotic normality (LAN) and local asymptotic mixed normality
(LAMN) and Cram\'{e}r style theory involving derivatives and Fisher information.
The main tool is convergence in law of the log likelihood function and its
derivatives considered as random elements of a Polish space of continuous
functions with the metric of uniform
convergence on compact sets.  We obtain results for both one-step-Newton
estimators and Newton-iterated-to-convergence estimators.
\end{abstract}



\end{frontmatter}

\section{Disclaimer}

Few papers need a disclaimer, but this paper seems to offend
some people.  For the record, your humble author has written papers using
conventional asymptotics before this paper and will continue to write such
papers in the future.  Many readers find this paper presents an interesting
new (to them) way to look at maximum likelihood, but some readers get upset
about what they take to be its disparagment of conventional asymptotics.
But this ``disparagement'' says nothing that every statistician does
not already know.  In any actual data analysis, the sample size is
not changing.  Whether or not more data of the same type may be collected
in the future has no bearing on the validity of the analysis being done now.
Whether asymptotics ``works'' or not for any actual data, cannot be established
by checking conditions of theorems, since the theorems only say asymptotics
``works'' for some sufficiently large $n$ not necessarily at the $n$ for the
actual data.  The only practical way to check whether asymptotics ``works''
is to do simulations.

\section{Introduction}

The asymptotics of maximum likelihood is beautiful theory.
If you can calculate two derivatives of the log likelihood,
you can find the maximum likelihood estimate (MLE) and its
asymptotic normal distribution.

But why does this work?  And when does it work?
The literature contains many different treatments, many theorems with
long lists of assumptions, each slightly different from the others,
and long messy calculations in their proofs.
But they give no insight because neither
assumptions nor calculations are sharp.
The beauty of the theory is hidden by the mess.

In this article we explore an elegant theory of the asymptotics
of likelihood inference inspired by the work of Lucien Le Cam,
presented in full generality in \citet{lecam-big} and in somewhat
simplified form in \citet{lecam-little}.
Most statisticians find even the simplified
version abstract and difficult.
Here we attempt to bring this theory down to a level most statisticians
can understand.  We take from Le Cam two simple ideas.
\begin{itemize}
\item If the log likelihood is approximately quadratic with constant Hessian,
    then the MLE is approximately normal.
    This is the theory of locally asymptotically normal (LAN) and
    locally asymptotically mixed normal (LAMN) models.
\item Asymptotic theory does not need $n$ going to infinity.  We can
    dispense with sequences of models, and instead compare
    the actual model for the actual data to an LAN or LAMN model.
\end{itemize}

Although these ideas are not new, being ``well known'' to a handful
of theoreticians, they are not widely understood.  When I tell
typical statisticians that the asymptotics of maximum likelihood have
nothing to do with the law of large numbers (LLN) or
the central limit theorem (CLT) but rather with how close the log
likelihood is to quadratic, I usually get blank stares.  That's not
what they learned in their theory class.  Worse, many statisticians
have come to the conclusion that since the ``$n$ goes to infinity''
story makes no sense, asymptotics are bunk.  It is a shame that
so many statisticians are so confused about a crucial aspect of statistics.

\subsection{Local Asymptotic Normality} \label{sec:lan-intro}

Chapter~6 of \citet{lecam-little} presents the theory of
local asymptotically normal (LAN) and locally asymptotically mixed normal
(LAMN) sequences of models.  This theory has no assumption of independent and
identically distributed data (nor even stationarity and weak dependence).
The $n$ indexing an LAN or LAMN sequence of models is just an index
that need not have anything to do with sample size or anything analogous
to sample size.  Thus the LLN and the CLT
do not apply, and asymptotic normality must arise from some other
source.

Surprising to those who haven't seen it before,
asymptotic normality arises from the log likelihood being
asymptotically quadratic.  A statistical model with exactly quadratic
log likelihood
\begin{equation} \label{eq:logl}
   l(\theta)
   =
   U + Z' \theta - \tfrac{1}{2} \theta' K \theta,
   \qquad \theta \in \real^p
\end{equation}
where $U$ is a random scalar, $Z$ is a random $p$ vector, and $K$
is a random $p \times p$ symmetric matrix, has observed Fisher
information $K$ and maximum likelihood estimator (MLE)
\begin{equation} \label{eq:mqle}
   \hat{\theta} = K^{-1} Z
\end{equation}
(if $K$ is positive definite).
If $K$ is constant, then the distribution of the MLE is
\begin{equation} \label{eq:dist-mle}
   \hat{\theta} \sim \NormalDis(\theta, K^{-1}),
\end{equation}
the right hand side being the multivariate normal distribution with
mean vector $\theta$ and variance matrix $K^{-1}$
(Corollary~\ref{cor:quad} below).
If $K$ is random, but \emph{invariant in law}, meaning its distribution does
not depend on the parameter $\theta$, then \eqref{eq:dist-mle} holds
conditional on $K$ (Theorem~\ref{th:quad} below).

The essence of likelihood asymptotics consists of the idea that if
the log likelihood is only approximately of the form \eqref{eq:logl}
with $K$ invariant in law, then \eqref{eq:dist-mle} holds approximately
too.  All likelihood asymptotics that produce conclusions resembling
\eqref{eq:dist-mle} are formalizations of this idea.  The idea may
be lost in messy proofs, but
it's what really makes likelihood inference work the way it does.

\subsection{``No \protect{$n$}'' Asymptotics} \label{sec:no-n-intro}

By ``no $n$ asymptotics'' we mean asymptotics done without reference
to any sequence of statistical models.  There is no $n$ going to infinity.
This is not a new idea,
as the following quotation from \citet[p.~xiv]{lecam-big} shows.
\begin{quotation}
From time to time results are stated as limit theorems obtainable as
something called $n$ ``tends to infinity.''  This is especially so in
Chapter~7 [Some Limit Theorems] where the results are just limit theorems.
Otherwise we have
made a special effort to state the results in such a way that they could
eventually be transformed into approximation results.  Indeed, limit theorems
``as $n$ tends to infinity'' are logically devoid of content about what happens
at any particular $n$.  All they can do is suggest certain approaches whose
performance must then be checked on the case at hand.  Unfortunately the
approximation bounds we could get were too often too crude and cumbersome
to be of any practical use.  Thus we have let $n$ tend to infinity, but
we urge the reader to think of the material in approximation terms, especially
in subjects such as ones described in Chapter~11
[Asymptotic Normality---Global].
\end{quotation}

Le Cam's point that asymptotics ``are logically devoid of content about
what happens at any particular $n$'' refers to the fact that convergence
of a sequence tells us nothing about any initial segment of the sequence.
An estimator $\hat{\theta}_n$ that is equal to 42 for all $n < 10^{10}$
and equal to the MLE thereafter is asymptotically equivalent to the MLE.
Strictly speaking, asymptotics---at least the usual story about $n$ going
to infinity---does not distinguish between these estimators
and says you might just as well use one as the other.

The story about $n$ going to infinity is even less plausible in
spatial statistics and statistical genetics where every component
of the data may be correlated with every other component.  Suppose
we have data on school districts of Minnesota.
How does Minnesota go to infinity?  By invasion
of surrounding states and provinces of Canada, not to mention Lake Superior,
and eventually by rocket ships to outer space?
How silly does the $n$ goes to infinity
story have to be before it provokes laughter instead of reverence?

Having once seen the absurdity of the $n$ goes to infinity story in any
context, it becomes hard to maintain its illusion of appropriateness
in any other context.  A convert to the ``no $n$'' view always thinks
$n = 1$.  You always have one ``data set,'' which comprises the data
for an analysis.
And it isn't going to infinity or anywhere else.

But how do ``no $n$ asymptotics'' work?
It is not clear (to me) what Le Cam meant by ``eventually transformed
into approximation results'' because he
used a convergence structure
\citep[Chapter~6, Definitions 1, 2, and~3]{lecam-little}
so weak that, although topological \citep[Proposition~1.7.15]{beattie-butzmann},
it seems not metrizable.
So at this point I part company with Le Cam
and present my own take on ``no $n$ asymptotics''
(although our Section~\ref{sec:lan} follows Le Cam closely,
our Section~\ref{sec:approx} and Appendix~\ref{app:newton} seem new).
It has the following simple logic.
\begin{itemize}
\item All assumptions are packed into a single convergence in law
    statement involving the log likelihood.
\item The conclusion is a convergence in law statement about an estimator.
\item Hence delta and epsilon arguments using metrics for convergence
    in law can replace sequential arguments.
\end{itemize}

One source of this scheme is \citet{geyer-cons}, which did not use the whole
scheme, but which in hindsight should have.  The conclusion
of Lemma~4.1 in that article
is a ``single convergence in law statement about the log likelihood''
that incorporates most of the assumptions in that article.  It is a forerunner
of our treatment here.

At this point we introduce a caveat.  By ``no $n$'' asymptotics we do
not mean the letter $n$ is not allowed to appear anywhere, nor do we
mean that sequences are not allowed to be used.  We mean that we do
not use sequences associated with a story about the index $n$ being
sample size or anything analogous to sample size.

We do use sequences as a purely technical tool in discussing continuity.
To prove a function $f : U \to V$ between metric spaces with
metrics $d_U$ and $d_V$ continuous, one can show
that for every $x \in U$ and every $\epsilon > 0$ there exists a $\delta$,
which may depend on $x$ and $\epsilon$, such that
$d_V\bigl( f(x), f(y) \bigr) < \epsilon$,
whenever $d_U(x, y) < \delta$.
Alternatively, one can show that $f(x_n) \to f(x)$,
whenever $x_n \to x$.
The mathematical content is the same either way.

In particular, it is more convenient to use sequences in discussing
convergence of probability measures on Polish spaces.  The space
of all probability measures on a Polish space is itself a metric
space \citep[p.~72]{billingsley}.  Hence delta-epsilon arguments
can be used to discuss convergence, but nearly the whole probability
literature on convergence of probability measures uses sequences
rather than delta-epsilon arguments, so it is easier to cite the
literature if we use sequences.
The mathematical content is the same either way.

The point is that we do not reify $n$ as sample size.
So perhaps a longer name such as ``$n$ not reified as sample size'' asymptotics
would be better than the shorter ``no $n$'' asymptotics, but we will continue
with the shorter name.

What is important is ``quadraticity.''  If the log
likelihood for the actual model for the actual data is nearly quadratic,
then the MLE has the familiar properties
discussed in Section~\ref{sec:lan-intro}, but no story
about sample size going to infinity will make the actual log likelihood
more quadratic than it actually is or the actual MLE more nearly normal.

I concede that metrics for convergence in law are unwieldy and
might also give ``approximation bounds $\ldots$ too crude and cumbersome
to be of any practical use.''  But, unlike Le Cam, I am not bothered by
this because of my familiarity with computer intensive statistical methods.

\subsection{Asymptotics is Only a Heuristic}

(This slogan means what Le Cam meant
by ``all they can do is suggest certain approaches''.)
We know that asymptotics often works well in practical problems because
we can check the asymptotics by computer simulation (perhaps what Le~Cam meant
by ``checked on the case at hand''), but conventional
theory doesn't tell us why asymptotics works when it does.  It only
tells us that asymptotics works for sufficiently large $n$, perhaps
astronomically larger than the actual $n$ of the actual data.  So that leaves
a theoretical puzzle.
\begin{itemize}
\item Asymptotics often works.
\item But it doesn't work for the reasons given in proofs.
\item It works for reasons too complicated for theory to handle.
\end{itemize}

I am not sure about ``too complicated $\ldots$ to handle.''
Perhaps a computer-assisted proof could give ``approximation
bounds $\ldots$ of practical use,'' what Le Cam wanted but could not
devise.  But when I think how much computer time a computer-assisted proof
might take and consider alternative ways to spend the computer time, I do
not see how approximation bounds could be as useful as a parametric bootstrap,
much less a double parametric bootstrap (since we are interested in
likelihood theory for parametric models we consider only the parametric
bootstrap).

A good approximation bound, even if such could be found, would only indicate
whether the asymptotics work or don't work, but a bootstrap
of approximately pivotal quantities derived from the asymptotics
not only diagnoses any
failure of the asymptotics but also provides a correction, so the
bootstrap may work when asymptotics fails.  And the double bootstrap
diagnoses any failure of the single bootstrap and provides further
correction, so the double bootstrap may work when the single bootstrap
fails \citep{beran-conf,beran-test,geyer-boot,hall,recycle}.

With ubiquitous fast computing, there is no excuse for not using
the bootstrap to improve the accuracy
of asymptotics in every serious application.
Thus we arrive at the following attitude about asymptotics
\begin{itemize}
\item Asymptotics is only a heuristic.  It provides no guarantees.
\item If worried about the asymptotics, bootstrap!
\item If worried about the bootstrap, iterate the bootstrap!
\end{itemize}

However, the only justification of the bootstrap is asymptotic.
So this leaves us in a quandary of circularity.
\begin{itemize}
\item The bootstrap is only a heuristic.  It provides no guarantees.
\item All justification for the bootstrap is asymptotic!
\item In order for the bootstrap to work well, one must bootstrap
    approximately asymptotically pivotal quantities!
\end{itemize}
(the ``approximately'' recognizes that something
less than perfectly pivotal, for example merely variance stabilized,
is still worth using).
% But this ``quandary'' is nothing but the central problem of
% frequentist statistics:
% if estimates are sloppy, estimates of their sampling distributions
% can be much worse.

In practice, this ``circularity'' does not hamper analysis.  In order to
devise good estimates one uses the asymptotics heuristic
(choosing the MLE perhaps).  In order
to devise ``approximately asymptotically pivotal
quantities'' one again uses the asymptotics heuristic (choosing log likelihood
ratios perhaps).
But when one calculates probabilities for tests and confidence intervals
by simulation, the calculation can be made arbitrarily accurate
for any given $\theta$.  Thus the traditional role of asymptotics,
approximating $P_\theta$, is not needed when we bootstrap.
We only need asymptotics to deal with the dependence
of $P_\theta$ on $\theta$.  Generally, this dependence never goes entirely
away, no matter how many times the bootstrap is iterated, but it does decrease
\citep{beran-conf,beran-test,hall}.

The parametric bootstrap simulates multiple data
sets $y^*_1$, $\ldots$, $y^*_n$ from
$P_{\hat{\theta}(y)}$, where $y$ is the real data
and $\hat{\theta}$ some estimator.  The
double parametric bootstrap
simulates multiple data sets from each $P_{\hat{\theta}(y^*_i)}$.
Its work load can be reduced by using importance sampling \citep{recycle}.
Assuming $\theta \mapsto P_\theta$
is continuous, % (and if it isn't one definitely needs to reparameterize),
these simulations tell everything about the model for $\theta$ in
the region filled out by the $\hat{\theta}(y^*_i)$.
But nothing prevents one from simulating from $\theta$ in a bigger region
if one wants to.

The parametric bootstrap is very different from the nonparametric bootstrap
in this respect.  The nonparametric bootstrap is inherently an asymptotic
(large $n$) methodology because resampling the data, which in effect
substitutes the empirical distribution of the data for the true unknown
distribution, only ``works'' when the empirical distribution is close to
the true distribution, which is when the sample size is very large.
When the parametric bootstrap simulates a distribution $P_\theta$ it does
so with error that does not depend on sample size but only on how long
we run the computer.  When a double parametric bootstrap simulates $P_\theta$
for many different $\theta$ values, we learn about $P_\theta$ for these
$\theta$ values and nearby $\theta$ values.  If these $\theta$ values are
densely spread over a region, we learn about $P_\theta$ for all $\theta$
in the region.  So long as the true unknown $\theta$ value lies in that
region, we approximate the true unknown distribution with accuracy that
does not depend on sample size.
Thus if one does enough simulation, the parametric bootstrap can be made
to work for small sample sizes in a way the nonparametric bootstrap cannot.
Since this is not a paper about the bootstrap, we will say no more on the
subject.

Nevertheless, asymptotics often does ``work'' and permits simpler calculations
while providing more insight.  In such cases,
the bootstrap when used as a diagnostic
\citep{geyer-boot} proves its own pointlessness.  A single bootstrap
often shows that it cannot improve the answer provided by asymptotics,
and a double bootstrap often shows that it cannot improve the answer
provided by the single bootstrap.

Since asymptotics is ``only a heuristic,'' the only interesting question
is what form of asymptotics provides the most useful heuristic and does so
in the simplest fashion.  This article is my attempt at an answer.

\subsection{An Example from Spatial Statistics}
\label{sec:strauss}

\citet{geyer-spat-sim} give a method of simulating spatial point processes
and doing maximum likelihood estimation.  In their example, a Strauss
process \citep{strauss}, they noted that the asymptotic distribution of
the MLE appeared to be very close to normal, although the best asymptotic
results they were able to find in the literature \citep{jensen-91,jensen-93}
only applied to Strauss processes with very weak dependence, hence
very close to a Poisson process \cite[Discussion]{geyer-spat-sim},
which unfortunately did not include their example.

From a ``no $n$'' point of view, this example is trivial.
A Strauss process is a two-parameter full exponential family.
In the canonical parameterization, which \citet{geyer-spat-sim}
were using, the random part of the log likelihood is linear in the
parameter, hence the Hessian is nonrandom.
Hence, according to the theory developed here, the MLE will be approximately
normal so long as the Hessian is approximately constant over the region
of parameter values containing most of the sampling distribution of the MLE.

\citet{geyer-spat-sim} did not then understand the ``no $n$'' view
and made no attempt at such verification (although it would have been easy).
Direct verification of quadraticity is actually unnecessary here,
because the curved part of the log likelihood (in the canonical
parameterization) of an exponential family is proportional to
the cumulant generating function of the canonical statistic,
so the family is nearly LAN precisely when
the distribution of the canonical statistic is nearly normal,
which \citet{geyer-spat-sim} did investigate (their Figure~2).

Thus there is no need for any discussion of anything going to infinity.
The asymptotics here, properly understood, are quite simple.
As we used to say back in the sixties, the ``$n$ goes to infinity''
story is part of the problem not part of the solution.

Although the exponential family aspect makes things especially simple,
the same sort of thing is true in general.
When the Hessian is random, it is enough that it be
nearly invariant in law.

\subsection{An Example from Statistical Genetics}
\label{sec:fisher}

\citet{fish18} proposed a model for quantitative genetics
that has been widely used in animal breeding and other areas
of biology and indirectly led to modern regression and analysis of variance.
The data $Y$ are multivariate normal, measurements of some quantitative
trait on individuals, decomposed as
$$
   Y = \mu + B + E
$$
where $\mu$ is an unknown scalar parameter, where $B$ and $E$
are independent multivariate normal, $\NormalDis(0, \sigma^2 A)$
and $\NormalDis(0, \tau^2 I)$, respectively,
where $\sigma^2$ and $\tau^2$ are unknown parameters
called the \emph{additive genetic} and \emph{environmental} variance,
respectively, $A$ is a known matrix called
the \emph{numerator relationship matrix} in the animal breeding literature
\citep{henderson}, and $I$ is the identity matrix.
The matrix $A$ is determined solely by the relatedness of the individuals
(relative to a known pedigree).  Every element of $A$ is nonzero if every
individual is related to every other, and this implies all components of
$Y$ are correlated.

In modern terminology, this is a mixed model with
fixed effect vector $\mu$ and variance components $\sigma^2$ and $\tau^2$,
but if the pedigree is haphazard so $A$ has no regular structure, there
is no scope for telling ``$n$ goes to infinity'' stories like \citet{miller}
does for mixed models for simple designed experiments.  Our ``no $n$''
asymptotics do apply.  The log likelihood may be nearly quadratic,
in which case we have the ``usual'' asymptotics.

The supplementary web site
\url{http://purl.umn.edu/92198} at the University of Minnesota Digital
Conservancy
gives details of two examples with 500 and 2000 individuals.
An interesting aspect of our approach is that its intimate
connection with Newton's method
(Section~\ref{sec:newton} and Appendix~\ref{app:newton})
forced us to find a good starting point for Newton's method
(a technology transfer from spatial statistics) and investigate
its quality.  Thus our theory helped improve methodology.

Readers wanting extensive detail must visit the web site.
The short summary is that the example with 500 individuals
is not quite in asymptopia,
the parametric bootstrap being needed for bias correction in
constructing a confidence interval for logit heritability
($\log \sigma^2 - \log \tau^2$).  But when this example
was redone with 2000 individuals, the bias problem went away
and the bootstrap could not improve asymptotics.

\subsection{Our Regularity Conditions}

Famously, Le Cam, although spending much effort on likelihood,
did not like \emph{maximum likelihood}.  \cite{lecam-phooey} gives
many examples of the failure of maximum likelihood.  Some are
genuine examples of bad behavior of the MLE.
Others can be seen as problems with the ``$n$ goes to infinity''
story as much as with maximum likelihood.  I have always thought
that article failed to mention Le Cam's main reason for dislike
of maximum likelihood: his enthusiasm for weakest possible regularity
conditions.  He preferred conditions so weak that nothing can be
proved about the MLE and other estimators must be used instead
\cite[Section~6.3]{lecam-little}.

His approach does allow the ``usual asymptotics of maximum likelihood''
to be carried over to quite pathological
models \cite[Example~7.1]{lecam-little}
but only by replacing the MLE with a different estimator.  The problem with
this approach (as I see it) is that the resulting theory no longer describes
the MLE, hence is no longer useful to applied statisticians.
(Of course, it would be useful if applied statisticians used such pathological
models and such estimators.  As far as I know, they don't.)

Thus we stick with old-fashioned regularity conditions involving derivatives
of the log likelihood that go back to \citet[Chapters~32 and~33]{cramer}.
We shall investigate the consequences of being ``nearly'' LAMN in the sense
that the log likelihood and its first two derivatives are near those
of an LAMN model.  These conditions are about the weakest that still permit
treatment of Newton's method, so our theorems apply to the way maximum
likelihood is done in practice.
Our approach has the additional benefit of making our theory no stranger
than it has to be in the eyes of a typical statistician.

\section{Models with Quadratic Log Likelihood} \label{sec:lan}

\subsection{Log Likelihood}

The log likelihood for a parametric family of probability distributions
having densities $f_\theta$, $\theta \in \Theta$ with respect to
a measure $\lambda$ is a random function $l$ defined by
\begin{equation} \label{eq:logl-actual}
   l(\theta) = u(X) + \log f_\theta(X), \qquad \theta \in \Theta,
\end{equation}
where $X$ is the random data for the problem and $u$ is any real valued
function on the sample space that does not depend on the parameter $\theta$.
% In general, it is possible for $f_\theta(x)$ to be zero, and we need the
% convention
% $\log(0) = - \infty$ in order for the log likelihood to be well defined.
In this article, we are only interested families of almost surely positive
densities so the argument of the logarithm in \eqref{eq:logl-actual} is
never zero and the log likelihood is well defined.
This means all the distributions in the family
are absolutely continuous with respect to each other.

Then for any bounded random variable $g(X)$ and any parameter
values $\theta$ and $\theta + \delta$ we can write
\begin{equation} \label{eq:import}
\begin{split}
   E_{\theta + \delta}\{ g(X) \}
   & =
   \int g(x) f_{\theta + \delta}(x) \lambda(d x)
   \\
   & =
   \int g(x) \frac{f_{\theta + \delta}(x)}{f_\theta(x)} f_\theta(x) \lambda(d x)
   \\
   & =
   \int g(x) e^{l({\theta + \delta}) - l(\theta)} f_\theta(x) \lambda(d x)
   \\
   & =
   E_\theta\bigl\{ g(X) e^{l({\theta + \delta}) - l(\theta)} \bigr\}
\end{split}
\end{equation}
The assumption of almost surely positive densities is crucial.
Without it, the second line might not make sense because of
division by zero.

\subsection{Quadratic Log Likelihood}

Suppose the log likelihood is defined by \eqref{eq:logl}.
  The random variables
$U$, $Z$, and $K$ are, of course, functions of the data for the model,
although the notation does not indicate this explicitly.

The constant term $U$ in \eqref{eq:logl} analogous to
the term $u(X)$ in \eqref{eq:logl-actual} is of
no importance.  We are mainly interested in log likelihood ratios
\begin{equation} \label{eq:logl-rat}
   l(\theta + \delta) - l(\theta)
   =
   (Z - K \theta)' \delta - \tfrac{1}{2} \delta' K \delta,
\end{equation}
in which $U$ does not appear.

\pagebreak[3]
\begin{theorem}[LAMN] \label{th:quad}
Suppose \eqref{eq:logl} is the log likelihood of a probability model, then
\begin{enumerate}
\item[\upshape (a)] $K$ is almost surely positive semi-definite.
\end{enumerate}
Also, the following two conditions are equivalent (each implies the other).
\begin{enumerate}
\item[\upshape (b)] The conditional distribution of $Z$ given $K$
    for parameter value $\theta$
    is $\NormalDis(K \theta, K)$.
\item[\upshape (c)] The distribution of $K$ does not depend on the parameter
    $\theta$.
\end{enumerate}
\end{theorem}
Any model satisfying the conditions of the theorem is said to be LAMN
(locally asymptotically mixed normal).  Strictly speaking,
``locally asymptotically'' refers to sequences converging to such a model,
such as those discussed in Section~\ref{sec:approx},
but ``MN'' by itself is too short to make a good acronym and LAMN is standard
in the literature.

Our theorem is much simpler than
traditional LAMN theorems \citep[Lemmas~6.1 and~6.3]{lecam-little}
because ours is not asymptotic and the traditional ones are.
But the ideas are the same.
\begin{proof}
The special case of \eqref{eq:import} where $g$ is identically equal to one
with \eqref{eq:logl-rat} plugged in gives
\begin{equation} \label{eq:fred}
   1
   =
   E_{\theta + \delta}(1)
   =
   E_\theta\bigl\{ e^{(Z - K \theta)' \delta - \tfrac{1}{2} \delta' K \delta}
   \bigr\}
\end{equation}
Averaging \eqref{eq:fred} and \eqref{eq:fred} with $\delta$ replaced by
$- \delta$ gives
\begin{equation} \label{eq:quad:foo}
   E_\theta\bigl\{ \cosh[(Z - K \theta)' \delta]
   \exp[- \tfrac{1}{2} \delta' K \delta]
   \bigr\}
   =
   1.
\end{equation}

Now plug in $s \delta$ for $\delta$ in \eqref{eq:quad:foo},
where $s$ is scalar, and use the fact that
the hyperbolic cosine is always greater than one giving
\begin{align*}
   1
   & =
   E_\theta\bigl\{ \cosh[s (Z - K \theta)' \delta]
   \exp[- \tfrac{1}{2} s^2 \delta' K \delta]
   \bigr\}
   \\
   & \ge
   E_\theta\bigl\{ \exp[- \tfrac{1}{2} s^2 \delta' K \delta]
   I_{(-\infty, -\epsilon)}(\delta' K \delta)
   \bigr\}
   \\
   & \ge
   \exp(\tfrac{1}{2} s^2 \epsilon)
   P_\theta\bigl\{
   \delta' K \delta < - \epsilon
   \bigr\}
\end{align*}
For any $\epsilon > 0$,
the first term on the right hand side goes to infinity as
$s \to \infty$.  Hence the second term on the right hand side
must be zero.  Thus
$$
   P_{\theta}\{ \delta' K \delta < - \epsilon \} = 0, \qquad \epsilon > 0
$$
and continuity of probability implies the equality holds for $\epsilon = 0$
as well.  This proves (a).

Replace $g(X)$ in \eqref{eq:import} by $h(K)$ where $h$
is any bounded measurable function giving
\begin{equation} \label{eq:z-1}
   E_{\theta + \delta}\{h(K)\}
   =
   E_\theta\bigl\{ h(K)
   e^{(Z - K \theta)' \delta - \tfrac{1}{2} \delta' K \delta}
   \bigr\}.
\end{equation}

Assume (b).  Then the
moment generating function of $Z$ given $K$ for the parameter $\theta$
is
\begin{equation} \label{eq:z-3}
   E_\theta\bigl\{ e^{Z' \delta} \bigm| K \bigr\}
   =
   e^{ \theta' K \delta + \tfrac{1}{2} \delta' K \delta}
\end{equation}
and this implies
\begin{equation} \label{eq:z-2}
   E_\theta\bigl\{
   e^{(Z - K \theta)' \delta - \tfrac{1}{2} \delta' K \delta}
   \bigm| K \bigr\}
   =
   1.
\end{equation}
Plugging \eqref{eq:z-2} into \eqref{eq:z-1} and using the iterated
expectation theorem we get
\begin{align*}
   E_{\theta + \delta}\{h(K)\}
   & =
   E_\theta\Bigl\{ h(K)
   E_\theta\bigl\{
   e^{(Z - K \theta)' \delta - \tfrac{1}{2} \delta' K \delta}
   \bigm| K
   \bigr\}
   \Bigr\}
   \\
   & =
   E_\theta\{h(K)\}
\end{align*}
which, $h$ being arbitrary, implies (c).  This proves (b) implies (c).

Now drop the assumption of (b) and assume (c), which implies the
left hand side of \eqref{eq:z-1} does not depend on $\delta$, hence
\begin{equation} \label{eq:z-1.5}
   E_\theta\{h(K)\}
   =
   E_\theta\bigl\{ h(K)
   e^{(Z - K \theta)' \delta - \tfrac{1}{2} \delta' K \delta}
   \bigr\}.
\end{equation}
By the definition of conditional expectation \eqref{eq:z-1.5} holding
for all bounded measurable functions $h$ implies
\eqref{eq:z-2} and hence \eqref{eq:z-3}, which implies (b).
This proves (c) implies (b).
\end{proof}

\begin{corollary}[LAN] \label{cor:quad}
Suppose \eqref{eq:logl} is the log likelihood of an identifiable
probability model and $K$ is constant, then
\begin{enumerate}
\item[\upshape (a)] $K$ is positive definite.
\end{enumerate}
Moreover,
\begin{enumerate}
\item[\upshape (b)] The distribution of $Z$
    for the parameter value $\theta$
    is $\NormalDis(K \theta, K)$.
\end{enumerate}
\end{corollary}
Any model satisfying the conditions of the corollary is said to be LAN
(locally asymptotically normal).  Strictly speaking,
as we said about LAMN,
the ``locally asymptotically'' refers to sequences converging to such a model,
but we also use it for the model itself.
\begin{proof}
The theorem applied to the case of constant $K$ gives (a) and (b)
with ``positive definite'' in (a) replaced by ``positive semi-definite.''  So
we only need to prove that identifiability implies positive definiteness.

If $K$ were not positive definite,
there would be a nonzero $\delta$ such that $\delta' K \delta = 0$, but this
would imply $K \delta = 0$ and the
distribution of $Z$ for the parameter value $\theta + \delta$
would be $\NormalDis(K \theta, K)$.  Hence the model would not be identifiable
(since $Z$ is a sufficient statistic, if the distribution of $Z$ is not
identifiable, neither is the model).
\end{proof}

We cannot prove the analogous property, $K$ almost surely positive definite,
for LAMN.  So we will henceforth assume it.  (That this doesn't
follow from identifiability is more a defect in the notion of identifiability
than in LAMN models.)

\subsection{Examples and Non-Examples}

The theorem provides many examples of LAMN.  Let $K$ have any distribution
that is almost surely positive-definite-valued, a Wishart distribution,
for example.  Then let $Z \mid K$ be $\NormalDis(K \theta, K)$.

The corollary provides a more restricted range of examples of LAN.  They
are the multivariate normal location models with nonsingular variance matrix
that does not depend on the parameter.

A non-example is
the AR(1) autoregressive model with known innovation
variance and unknown autoregressive parameter.  Let $X_0$ have any
distribution not depending on the parameter $\theta$, and recursively
define
\begin{equation} \label{eq:ar}
   X_n = \theta X_{n - 1} + Z_n
\end{equation}
where the $Z_i$ are independent and identically $\NormalDis(0, 1)$
distributed.  The log likelihood is
$$
   l_n(\theta) = - \tfrac{1}{2} \sum_{i = 1}^n (X_i - \theta X_{i - 1})^2
$$
which is clearly quadratic with Hessian
$- K_n = - \sum_{i = 1}^n X_{i - 1}^2$.
From \eqref{eq:ar}
$$
   E_\theta(X_n^2 | X_0) = \theta^2 E_\theta(X_{n - 1}^2 | X_0) + 1
$$
% so for $n > 0$
% \begin{align*}
%    E_\theta(X_n^2 | X_0)
%    & =
%    \theta^{2 n} X_0^2 + \sum_{k = 0}^{n - 1} \theta^{2 k}
%    \\
%    & =
%    \theta^{2 n} X_0^2 + \frac{1 - \theta^{2 n}}{1 - \theta^2}
% \end{align*}
%
% foo[n_] = theta^(2 n) xzero^2 + Sum[ theta^(2 k), {k, 0, n - 1} ]
% foo[n_] = Simplify[foo[n]]
% bar[n_] = Sum[ foo[i - 1], {i, 1, n - 1} ]
% bar[n_] = Simplify[bar[n]]
%
% foo[n] - (1 / (1 - theta^2) + theta^(2 n) (xzero^2 - 1 / (1 - theta^2)))
% Simplify[%]
% bar[n] - ((n - 1) / (1 - theta^2) + (xzero^2 - 1 / (1 - theta^2))
%     (1 - theta^(2 (n - 1))) / (1 - theta^2))
% Simplify[%]
%
and from this we can derive
$$
   E_\theta(K_n | X_0)
   =
%    \sum_{i = 1}^{n - 1} E_\theta(X_{i - 1}^2 | X_0)
%    =
   \frac{n - 1}{1 - \theta^2}
   +
   \left[ X_0^2 - \frac{1}{1 - \theta^2} \right]
   \frac{1 - \theta^{2 (n - 1)}}{1 - \theta^2},
   \qquad \theta \neq 1,
$$
which is enough to show
that the distribution of $K_n$ depends on $\theta$
so this model is not LAMN (in the ``no $n$'' sense we are using
the term in this article, although it is LAN in the limit as $n \to \infty$
for some values of $\theta$).

\section{Likelihood Approximation} \label{sec:approx}

It is plausible that a model that is ``nearly'' LAN has an
MLE that is ``nearly'' normally distributed and a model that
is ``nearly'' LAMN has an MLE that is ``nearly'' conditionally
normally distributed.  In order to make these vague statements
mathematically precise, we need to define what we mean by ``nearly''
and explore its mathematical consequences.

\subsection{Convergence in Law in Polish Spaces} \label{sec:polish}

Since log likelihoods are random functions, it makes sense to measure how
close one is to another in the sense of convergence in law.
In order to do that, we need a theory of convergence in law
for function-valued random variables.  We use the simplest
such theory: convergence in law for random variables
taking values in a Polish space (complete separable metric space).
\citet[first edition 1968]{billingsley} has the theory we will need.
Other sources
are \citet{fristedt} and \citet{shorack}.

A sequence of random elements $X_n$ of a Polish space $S$
\emph{converges in law} to another random element $X$ of $S$ if
$E \{ f(X_n) \} \to E \{ f(X) \}$, for every bounded continuous
function $f : S \to \real$.
This convergence is denoted
$$
   X_n \weakto X
$$
and is also called \emph{weak convergence},
a term from functional analysis, and (when the Polish space is $\real^p$)
\emph{convergence in distribution}.

The theory of convergence in law in Polish spaces is often considered
an advanced topic, but our use of it here involves
only
\begin{itemize}
\item the mapping theorem \citep[Theorem~2.7]{billingsley}
\item the Portmanteau theorem \citep[Theorem~2.1]{billingsley}
\item Slutsky's theorem \citep[Theorem~3.1]{billingsley}
\item Prohorov's theorem \citep[Theorems~5.1 and~5.2]{billingsley}
\item the subsequence principle \citep[Theorem~2.6]{billingsley}
\end{itemize}
which should be in the toolkit of every theoretical statistician.
They are no more difficult to use on random elements of Polish spaces
than on random vectors.  The last three of these are only used in
Appendix~\ref{app:usual} where we compare our theory to conventional
theory.  Only the mapping and Portmanteau theorems are needed to understand
the main text.

\subsection{The Polish Spaces \protect{$C(W)$ and $C^2(W)$}}
\label{sec:polish-funky}

Let $W$ be an open subset of $\real^p$ and
let $C(W)$ denote the space of all continuous real-valued 
functions on $W$.
The topology
of \emph{uniform convergence on compact sets} for $C(W)$
defines $f_n \to f$ if
$$
   \sup_{x \in B} \abs{f_n(x) - f(x)} \to 0,
   \qquad \text{for every compact subset $B$ of $W$}.
$$
The topology
of \emph{continuous convergence} for $C(W)$
defines $f_n \to f$ if 
$$
   f_n(x_n) \to f(x), \qquad \text{whenever $x_n \to x$}.
$$
\citet[Section~20, VIII, Theorem~2]{kuratowski} shows
these two topologies for $C(W)$ coincide.
\citet[Ch.~10, Sec.~3.3, Corollary, part (b)]{bourbaki}
shows that $C(W)$ is a Polish space.

Let $C^2(W)$ denote the space of all twice continuously differentiable
functions $W \to \real$ with the topology of uniform convergence on
compact sets (or continuous convergence)
of the functions and their first and second partial derivatives.
Then $C^2(W)$ is isomorphic to a subspace of
$C(W)^{1 + p + p \times p}$, in fact a closed subspace because
uniform convergence of derivatives
implies sequences can be differentiated term by term.
Hence $C^2(W)$ is a Polish space,
\citep[Propositions~2 and~3 of Section~18.1]{fristedt}.

We will consider an almost surely twice continuously differentiable
log likelihood whose parameter space is all of $\real^p$
to be a random element of $C^2(\real^p)$ and will restrict ourselves to
such log likelihoods
(for measurability issues, see Appendix~\ref{app:polish-funky}).
We find a way to work around
the assumption that the parameter space be all of $\real^p$
in Section~\ref{sec:bound}.

\subsection{Sequences of Statistical Models} \label{sec:key}

In order to discuss convergence we use sequences of models, but merely
as technical tools.  The $n$ indexing a sequence
need have nothing to do with sample size, and models in the sequence
need have nothing to do with each other except that they all have the
same parameter space, which is $\real^p$.
As we said in Section~\ref{sec:no-n-intro}, we could eliminate
sequences from the discussion if we wanted to.
The models have log likelihoods $l_n$ and true parameter values $\psi_n$.

Merely for comparison with conventional theory (Appendix~\ref{app:usual})
we also introduce a ``rate'' $\tau_n$.  In conventional asymptotic theory
$\tau_n = \sqrt{n}$ plays an important role,
so we put it in our treatment too.
In the ``no $n$'' view, however, where the models in the sequence
have ``nothing to do with each other'' there is no role for $\tau_n$
to play, and we set $\tau_n = 1$ for all $n$.

Define random functions $q_n$ by
\begin{equation} \label{eq:quoin}
   q_n(\delta) = l_n(\psi_n + \tau_n^{-1} \delta) - l_n(\psi_n),
   \qquad \delta \in \real^p.
\end{equation}
These are also log likelihoods, but we have changed the parameter
from $\theta$ to $\delta$, so the true value of the parameter $\delta$ is
zero, and subtracted a term not containing $\delta$,
so $q_n(0) = 0$.

Our key assumption is
\begin{equation} \label{eq:assume-key}
   q_n \weakto q, \qquad \text{in $C^2(\real^p)$}
\end{equation}
where $q$ is the log likelihood of an LAMN model
\begin{equation} \label{eq:q}
   q(\delta) = \delta' Z - \tfrac{1}{2} \delta' K \delta,
   \qquad \delta \in \real^p,
\end{equation}
there being no constant term in \eqref{eq:q} because $q_n(0) = 0$ for all $n$
implies $q(0) = 0$.
A consequence of LAMN is that $e^{q(\delta)}$ is for each $\delta$ a
probability density with respect to the measure governing the law of $q$,
which is the measure in the LAMN model for $\delta = 0$ because $e^{q(0)} = 1$.
This implies
\begin{equation} \label{eq:contiguity}
   E \bigl\{e^{q(\delta)}\bigr\} = 1, \qquad \delta \in \real^p.
\end{equation}

Please note that, despite our assuming LAMN with ``N'' standing for normal,
we can say we are not actually assuming asymptotic normality.  Asymptotic
(conditional) normality comes from the equivalence of the two conditions
in the LAMN theorem (Theorem~\ref{th:quad}).  We can say that we are assuming
condition (c) of the theorem and getting condition (b) as a consequence.
Normality arises here from the log likelihood being quadratic
and its Hessian being invariant in law.
Normality is not assumed, and the CLT plays no role.

\subsection{Contiguity}

In Le Cam's theory, property \eqref{eq:contiguity} is exceedingly
important, what is referred
to as \emph{contiguity} of the sequence of probability
measures having parameter values $\psi_n + \tau_n^{-1} \delta$ to the
sequence having parameter values $\psi_n$
\citep[Theorem~1 of Chapter~3]{lecam-little}.

Contiguity \eqref{eq:contiguity}
does not follow from the convergence in law \eqref{eq:assume-key}
by itself.  From that we can only
conclude by Fatou's lemma for convergence in law
\citep[Theorem~3.4]{billingsley}
\begin{equation} \label{eq:lose}
   E \bigl\{e^{q(\delta)}\bigr\} \le 1, \qquad \delta \in \real^p.
\end{equation}
That we have the equality \eqref{eq:contiguity} rather than the inequality
\eqref{eq:lose} is the contiguity property.
In our ``no $n$'' theory \eqref{eq:contiguity} arises naturally.
We always have it because ``improper'' asymptotic models (having
densities that don't integrate to one) make no sense.

\subsection{One Step of Newton's Method} \label{sec:newton}

Our presentation is closely tied to Newton's method.  This is a minor
theme in conventional theory (one-step Newton updates of
root-$n$-consistent estimators are asymptotically equivalent to the MLE),
but we make it a major theme, bringing our theory
closer to actual applications.   We consider two kinds of
estimator: one-step-Newton estimators, like those in conventional
theory, and Newton-iterated-to-convergence estimators
(Appendix~\ref{app:newton}), what in applications
are usually deemed MLE.

To treat these estimators we need to consider not only the log
likelihood but also the starting point for Newton's method.
The one-step Newton map $G$ is defined
for an objective function $q$ and a current iterate $\delta$ by
\begin{equation} \label{eq:newton-map}
   G(q, \delta)
   =
   \delta + \bigl(- \nabla^2 q(\delta)\bigr)^{-1} \nabla q(\delta)
\end{equation}
when $- \nabla^2 q(\delta)$ is positive definite (otherwise Newton's method
makes no sense as an attempt at maximization).
In order to deal with this possible failure of Newton's method, we
allow $\delta$ and $G(q, \delta)$ to have the value
NaO (not an object), which, like NaN (not a number) in computer arithmetic,
propagates through all operations.  Addition of a new isolated point
to a Polish space produces another Polish space.

\begin{lemma} \label{lem:one-lime}
If $q$ is strictly concave quadratic and $\delta \neq \text{\upshape NaO}$,
then $G(q, \delta)$ is the unique point where $q$ achieves its maximum.
\end{lemma}
\begin{proof}
Suppose
\begin{equation} \label{eq:lime}
   q(\delta) = u + z' \delta - \tfrac{1}{2} \delta' K \delta,
\end{equation}
where $K$ is positive definite, so
\begin{align*}
   \nabla q(\delta) & = z - K \delta
   \\
   \nabla^2 q(\delta) & = - K
\end{align*}
Then
$$
   G(q, \delta) = \delta + K^{-1} (z - K \delta) = K^{-1} z,
$$
which is the solution to $\nabla q(\delta) = 0$, and hence by
strict concavity, the unique global maximizer of $q$.
% For the second assertion, there is nothing to prove
% if $G^\infty(q, \delta) = \text{NaO}$.  Otherwise
% \eqref{eq:newton-limit-def} holds with $\delta_n \neq \text{NaO}$ for all $n$.
\end{proof}

\begin{lemma} \label{lem:one-lemon}
The Newton map $G$ is
continuous at points $(q, \delta)$ such that $q$ is strictly concave
quadratic.
\end{lemma}
% Saying that $q$ is quadratic is the same as saying that
% $\delta \mapsto \nabla^2 q(\delta)$ is constant.
% Saying that $q$ is strictly concave is the same as saying that
% $- \nabla^2 q(\delta)$ is positive definite at all $\delta$, but when $q$
% is quadratic this holds for all $\delta$ if it holds for even one $\delta$.
\begin{proof}
Suppose $\delta_n \to \delta$ and
$q_n \to q$ with $q$
given by \eqref{eq:lime}.
If $\delta = \text{NaO}$ the conclusion is trivial.
Otherwise, because convergence in $C^2(\real^p)$
implies continuous convergence of first and second derivatives
$$
   - \nabla^2 q_n(\delta_n) \to - \nabla^2 q(\delta) = K,
$$
which is positive definite, so $G(q_n, \delta_n) \neq \text{NaO}$
for sufficiently large $n$, and
$$
   \nabla q_n(\delta_n) \to \nabla q(\delta) = z - K \delta,
$$
so
$$
   G(q_n, \delta_n) \to G(q, \delta)
   =
   K^{-1} z.
$$
\end{proof}

\begin{theorem} \label{th:main}
Let $\tilde{\delta}_n$ be a sequence of random elements of $\real^d$
and $q_n$ a sequence of log likelihoods having parameter space $\real^d$
and true parameter value zero.  Suppose
\begin{equation} \label{eq:main}
   \bigl( q_n, \tilde{\delta}_n \bigr)
   \weakto
   (q, \tilde{\delta})
\end{equation}
where $q$ is the log likelihood of an LAMN model \eqref{eq:q},
and define $\hat{\delta}_n = G(q_n, \tilde{\delta}_n)$ where
$G$ is given by \eqref{eq:newton-map}.
Then
\begin{subequations}
\begin{align}
   - \nabla^2 q_n(\tilde{\delta}_n) & \weakto K
   \label{eq:main-a}
   \\
   - \nabla^2 q_n(\hat{\delta}_n) & \weakto K
   \label{eq:main-b}
   \\
   \hat{\delta}_n
   & \weakto
   K^{-1} Z
   \label{eq:main-d}
   \\
   \bigl( - \nabla^2 q_n(\hat{\delta}_n) \bigr)^{1 / 2}
   \hat{\delta}_n
   & \weakto
   \NormalDis(0, I)
   \label{eq:main-c}
\end{align}
\end{subequations}
where the matrix square root in \eqref{eq:main-c} is the
symmetric square root when the matrix is positive definite and {\upshape NaO}
otherwise.  Moreover, \eqref{eq:main}, \eqref{eq:main-a},
\eqref{eq:main-b}, \eqref{eq:main-d}, and \eqref{eq:main-c} hold jointly.
\end{theorem}
The convergence in law \eqref{eq:main} takes place in
$C^2(\real^p) \times \real^p$,
the product of Polish spaces being a Polish
space \citep[Appendix~M6]{billingsley}.
In \eqref{eq:main-c} $I$ is 
the $p \times p$ identity matrix.  Since the random variables
on the right hand sides of all these limits are never
NaO, the left hand sides are NaO with probability
converging to zero.

Some would rewrite \eqref{eq:main-c} as
\begin{equation} \label{eq:bogus}
   \hat{\theta}_n \approx
   \NormalDis\left(\psi_n,
   \bigl( - \nabla^2 l_n(\hat{\theta}_n) \bigr)^{- 1}
   \right),
\end{equation}
where the double wiggle means ``approximately distributed'' or something of
the sort and where we have shifted back to the original parameterization
\begin{align*}
    \tilde{\theta}_n & = \psi_n + \tau_n^{-1} \tilde{\delta}_n
    \\
    \hat{\theta}_n & = \psi_n + \tau_n^{-1} \hat{\delta}_n
    \\
                   & = G(l_n, \tilde{\theta}_n)
\end{align*}
Strictly speaking, \eqref{eq:bogus} is mathematical
nonsense, having no mathematical content except by allusion to
\eqref{eq:main-c}, but it is similar to
\begin{equation} \label{eq:bogosaurus}
   \hat{\theta}_n \approx
   \NormalDis\left(\psi,
   I_n(\hat{\theta}_n)^{- 1}
   \right)
\end{equation}
familiar from conventional treatments of the asymptotics of maximum likelihood,
where $\psi$ is the true parameter value and $I_n(\theta)$ is expected
Fisher information for sample size $n$.

\begin{proof}
We claim the map
\begin{equation} \label{eq:biggie}
   \begin{pmatrix} q \\ \delta \end{pmatrix}
   \mapsto
   \begin{pmatrix}
   q
   \\
   \delta
   \\
   \nabla^2 q(\delta)
   \\
   \nabla^2 q\bigl(G(q, \delta)\bigr)
   \\
   G(q, \delta)
   \\
   \Bigl(- \nabla^2 q\bigl(G(q, \delta)\bigr)\Bigr)^{1 / 2}
   G(q, \delta)
   \end{pmatrix}
\end{equation}
is continuous on $C^2(\real^d) \times \real^2$ at points where $q$
is strictly concave quadratic.
By definition of product topology, \eqref{eq:biggie} is continuous
if each component is continuous.  The first two components are also
continuous by definition of product topology,
the fifth by Lemmas~\ref{lem:one-lime} and~\ref{lem:one-lemon},
then the third and fourth by continuous convergence, and the sixth
by matrix multiplication being a continuous operation, and
matrix square root being a continuous operation
on the set of positive definite matrices \cite[Problem~7.2.18]{horn}
and continuous convergence.

Now all of the assertions of the theorem follow from
the mapping theorem \citep[Theorem~2.7]{billingsley},
the only non-obvious limit being \eqref{eq:main-c}, which
is clearly $K^{- 1 / 2} Z$.  Since $Z$ is $\NormalDis(0, K)$ given $K$
by $q$ being LAMN with true parameter value zero,
$K^{- 1 / 2} Z$ is $\NormalDis(0, I)$ conditional on $K$.
Since the conditional distribution does not depend on $K$,
it is also the unconditional distribution.
\end{proof}

\begin{corollary} \label{cor:main}
If $K = - \nabla^2 q(\delta)$ in the theorem is constant, then
\eqref{eq:main-d} can be replaced by
\begin{equation} \label{eq:main-d-const}
   \hat{\delta}_n
   \weakto
   \NormalDis(0, K^{-1})
\end{equation}
\end{corollary}
The proof is obvious.
Some would rewrite \eqref{eq:main-d-const} as
\begin{equation} \label{eq:bogus-const}
   \hat{\theta}_n \approx
   \NormalDis\left(\psi_n, (\tau_n^2 K)^{- 1} \right),
\end{equation}
from which we see that $K$ plays the role of expected Fisher information
for sample size one and $\tau_n^2 K$ plays the role of expected Fisher
information for sample size $n$, though $K$ isn't the expectation of anything
in our setup.

\subsection{Bounding the Approximation} \label{sec:bound}

Suppose we wish to use \eqref{eq:main-c} to produce a confidence region.
Another application of the mapping theorem gives
$$
   \hat{\delta}_n'
   \bigl( - \nabla^2 q_n(\hat{\delta}_n) \bigr)
   \hat{\delta}_n
   \weakto
   \ChiSqDis(p),
$$
where the right hand side denotes a chi-square random variable with $p$
degrees of freedom.  If $\kappa$ is the upper $\alpha$ quantile
of this distribution, then by the
Portmanteau theorem \citep[Theorem~2.1]{billingsley}
$$
   \liminf_{n \to \infty}
   \Pr\left\{
   \hat{\delta}_n'
   \bigl( - \nabla^2 q_n(\hat{\delta}_n) \bigr)
   \hat{\delta}_n
   <
   \kappa
   \right\} \ge 1 - \alpha
$$
and, mapped to the original parameterization, this tells us that
$$
   \bigset{ \theta \in \real^p :
   (\hat{\theta}_n - \theta)'
   \bigl( - \nabla^2 l_n(\hat{\theta}_n) \bigr)
   (\hat{\theta}_n - \theta)
   <
   \kappa
   }
$$
is a $1 - \alpha$ asymptotic confidence region for the true unknown
parameter value $\psi_n$.

So far so conventional, but now we want to do a finer analysis
using the joint convergence of $\tilde{\theta}_n$ and $\hat{\theta}_n$.
For any open set $W$ we get
\begin{multline} \label{eq:bounded}
   \liminf_{n \to \infty}
   \Pr\left\{
   \hat{\delta}_n'
   \bigl( - \nabla^2 q_n(\hat{\delta}_n) \bigr)
   \hat{\delta}_n
   <
   \kappa
   \opand
   \tilde{\delta}_n \in W
   \opand
   \hat{\delta}_n \in W
   \right\}
   \\
   \ge
   1 - \alpha - \Pr ( \tilde{\delta} \notin W \opor \hat{\delta} \notin W )
\end{multline}
(using Bonferroni on the right hand side),
and now we note that we only evaluate $\nabla q_n$ and $\nabla^2 q_n$
at $\tilde{\delta}_n$ and $\hat{\delta}_n$,
and hence for this statement to hold we only
need the convergence in law \eqref{eq:main} relative to $W$, that is,
for \eqref{eq:bounded} to hold we only need \eqref{eq:main} to hold
in $C^2(W) \times W$.

This section is our counterpart of the conventional ``regularity condition''
that the true parameter value be an interior point of the parameter space.
But here we see that if $\tilde{\theta}_n$ is a very bad estimator of $\psi_n$,
then $W$ may need to be very large in order to have
$\Pr ( \tilde{\delta} \notin W \opor \hat{\delta} \notin W )$
very small.
The conventional regularity condition vastly oversimplifies
what is really needed.

\section{Discussion} \label{sec:discuss}

It has not escaped our notice that the ``no $n$'' view advocated here
leaves no place for a lot of established statistical theory: Edgeworth
expansions, rates of convergence, the Bayes information criterion (BIC),
and the subsampling bootstrap, to mention just a few.  Can all of this
useful theory be replaced by some ``no $n$'' analog?  Only time will
tell.  Our goal is merely to explicate this ``no $n$'' view of maximum
likelihood.  We are not advocating
political correctness in asymptotics.

It does seem obvious that ``no $n$'' asymptotics based on LAN and LAMN
theory says nothing about situations where no likelihood is specified
(as in quasilikelihood, estimating equations, and much of nonparametrics
and robustness) or where the likelihood is incorrectly specified or where
the likelihood is correctly specified but the parameter is infinite-dimensional
(as in nonparametric maximum likelihood).  Thus it is unclear how this
``no $n$'' view can be extended to these areas.

Another classical result that does not transfer
to the ``no $n$'' worldview is the asymptotic efficiency of maximum likelihood.
In an LAN model, it is true that the MLE is the best equivariant-in-law
estimator \citep[Proposition~8.4]{vdv}, but James-Stein estimators
\citep{jse} show that the MLE is not the best estimator (where ``best''
means minimum mean squared error).  The ``no $n$'' view has no room
for other interpretations: if one likes James-Stein estimators,
then one cannot also consider the MLE asymptotically efficient (because LAN
models are already in asymptopia).  The argument leading
to Le Cam's almost sure convolution theorem \citep[Section~8.6]{vdv}
cannot be transferred to the ``no $n$'' world (because it would have to
prove the MLE as good as James-Stein estimators in LAN models,
and it isn't).

The application described in Section~\ref{sec:strauss} convinced us
of the usefulness of this ``no $n$'' view in spatial statistics,
statistical genetics, and other areas
where complicated dependence
makes difficult the invention of reasonable ``$n$ goes to infinity'' stories,
much less the proving of anything about them.
But having seen the usefulness of the ``no $n$'' view in any context,
one wants to use it in every context.  Having
understood the power of ``quadraticity'' as an explanatory
tool,'' many opportunities to use it arise.
When a user asks whether $n$ is ``large enough'' when the log likelihood
is nowhere near quadratic, now the answer is ``obviously not.''
When a user asks whether there is a need to reparameterize when the Wald
confidence regions go outside the parameter space, now the answer is
``obviously.''

We imagine some readers wondering whether our ideas are mere
``generalized abstract nonsense.''  Are we not essentially assuming what we
are trying to prove?  Where are all the deltas and epsilons and inequality
bashing that one expects in ``real'' real analysis?
We believe such ``inequality bashing'' should be kept separate from the
main argument, because it needlessly restricts the scope of the theory.
Le Cam thought the same, keeping separate the arguments of
Chapters~6 and 7 in \citet{lecam-little}.

For readers who want to see that kind of argument, we have provided
Lemma~\ref{lem:kumquat} in Appendix~\ref{app:usual} that says
our ``key assumption'' \eqref{eq:assume-key} is weaker than
the ``usual regularity conditions'' for maximum likelihood
\citep[Chapter~18]{ferguson} in all respects except the
requirement that the parameter space be all of $\real^p$, which
we worked around in Section~\ref{sec:bound}.
A similar lemma using a different Polish space with weaker topology is
Lemma~4.1 in \citet{geyer-cons}, where
a ``single convergence in law statement about the log likelihood''
(the conclusion of the lemma) is shown to follow from
complicated analytic regularity conditions
(the hypotheses of the lemma),
modeled after \citet[pp.~138 ff.]{pollard}.

Despite all the theory we present, our message is very simple:
if the log likelihood is nearly quadratic with Hessian nearly equivariant
in law, then the ``usual'' asymptotics of maximum likelihood hold.
The only point of all the theory is to show that conventional theory,
which does not support our message, can be replaced by theory that does.

\appendix

\section{Measurability in \protect{$C(W)$ and $C^2(W)$}}
\label{app:polish-funky}

Let $W$ be an open subset of $\real^p$, and let $B_n$ be an increasing
sequence of compact subsets of $W$ whose union is $W$.
Then \citet[Example~1.44]{rudin} gives an explicit metric for the
space $C(W)$ defined in section~\ref{sec:polish-funky}
\begin{equation} \label{eq:metric-rudin}
   d(f, g)
   =
   \max_{n \ge 1} \frac{2^{-n} \norm{f - g}_{B_n}}{1 + \norm{f - g}_{B_n}}
\end{equation}
where for any compact set $B$ we define
$$
   \norm{f}_B = \sup_{x \in B} \abs{f(x)}.
$$

A map $F : \Omega \to C(W)$ is
measurable if and only if its uncurried form
$(\omega, \theta) \mapsto F(\omega)(\theta)$
is Carath\'{e}odory, meaning
\begin{itemize}
\item $\omega \mapsto F(\omega)(\theta)$ is measurable for each fixed
    $\theta \in W$ and
\item $\theta \mapsto F(\omega)(\theta)$ is continuous for each fixed
    $\omega \in \Omega$
\end{itemize}
(\citealp[Example~1.44]{rudin};
\citealp[Corollary~4.23 and Theorem~4.54]{aliprantis}).
The isomorphism between $C^2(W)$ and a closed subspace of
$C(W)^{1 + p + p \times p}$ means a function from a measurable space
to $C^2(W)$ is
measurable if and only if it and its first and second partial derivatives
have Carath\'{e}odory uncurried forms.

\section{Newton Iterated to Convergence} \label{app:newton}

In this appendix we consider Newton's method iterated to convergence.
Of course Newton need not converge, so we have to deal with that issue.
Define Newton iterates
\begin{align*}
   \delta_0 & = \delta
   \\
   \delta_n & = G(q, \delta_{n - 1}), \qquad n > 0.
\end{align*}
If the sequence $\{ \delta_n \}$ converges, then we define
\begin{equation} \label{eq:newton-limit-def}
   G^\infty(q, \delta) = \lim_{n \to \infty} \delta_n
\end{equation}
and otherwise we define $G^\infty(q, \delta) = \text{NaO}$.
This function $G^\infty$ is called the infinite-step Newton map.

\begin{lemma} \label{lem:infty-lime}
If $q$ is strictly concave quadratic and $\delta \neq \text{\upshape NaO}$,
then $G^\infty(q, \delta)$ is the unique point where $q$ achieves its maximum.
\end{lemma}
\begin{proof}
By Lemma~\ref{lem:one-lime} the Newton sequence
converges in one step in this case and has the asserted properties.
\end{proof}

\begin{lemma} \label{lem:infty-lemon}
The infinite-step Newton map $G^\infty$ is
continuous at points $(q, \delta)$ such that $q$ is strictly concave
quadratic.
\end{lemma}
\begin{proof}
Suppose $\tilde{\delta}_n \to \tilde{\delta}$ and
$q_n \to q$, and suppose $q$ is strictly concave
quadratic given by \eqref{eq:lime}.  Let $\hat{\delta} = K^{-1} z$
be the unique point at which $q$ achieves its maximum, and
let $B$ be a compact convex set sufficiently large so that
it contains $\hat{\delta}$ and $\tilde{\delta}$
in its interior and
$$
   \sup_{\delta \in \partial B} q(\delta) < q(\tilde{\delta}),
$$
where $\partial B$ denotes the boundary of $B$.
Let $S$ be the unit sphere.  Then
$$
  (\delta, t) \mapsto t' \bigl( - \nabla^2 q_n(\delta) - K \bigr) t
$$
converges continuously to zero on $B \times S$.  Hence
\begin{subequations}
\begin{equation} \label{eq:whompa}
   \sup_{\substack{\delta \in B \\ t \in S}}
   \bigabs{ t' \bigl( - \nabla^2 q_n(\delta) - K \bigr) t }
   \to
   0.
\end{equation}
Also
\begin{equation} \label{eq:whompb}
   q_n(\hat{\delta}) - \sup_{\delta \in \partial B} q_n(\delta) 
   \to
   q(\hat{\delta}) - \sup_{\delta \in \partial B} q(\delta) 
\end{equation} 
and
\begin{equation} \label{eq:whompc}
   q_n(\tilde{\delta}_n) - \sup_{\delta \in \partial B} q_n(\delta) 
   \to
   q(\tilde{\delta}) - \sup_{\delta \in \partial B} q(\delta)
\end{equation} 
\end{subequations}
and the limits in both \eqref{eq:whompb} and \eqref{eq:whompc} are positive.

Choose $N$ such that for all $n \ge N$ the left hand side of \eqref{eq:whompa}
is strictly less than one-third of the smallest eigenvalue of $K$
and the left hand sides of both \eqref{eq:whompb} and \eqref{eq:whompc}
are positive.
For $n \ge N$, $- \nabla^2 q_n(\delta)$ is positive definite
for $\delta \in B$, hence $q_n$ is strictly concave on $B$.
Let $\hat{\delta}_n$ denote the (unique by strict concavity) point
at which $q_n$ achieves its maximum over $B$.
Since the left hand side of \eqref{eq:whompb} is positive,
$\hat{\delta}_n$ is in the interior of $B$
and $\nabla q_n(\hat{\delta}_n) = 0$.
By Theorem~1 and Example~2 of \citet{veselic} Newton's method applied
to $q_n$ starting at $\tilde{\delta}_n$ converges to $\hat{\delta}_n$
for $n \ge N$.
% all iterates of Newton's method being contained in
% the interior of $B$.

Consider the functions $H_n$ and $H$ defined by $H_n(\delta) = G(q_n, \delta)$
and $H(\delta) = G(q, \delta)$.  Then for $n \ge N$, we have
$\hat{\delta}_n = H_n(\hat{\delta}_n)$.  By Lemma~\ref{lem:one-lemon}
$H_n$ converges to $H$ uniformly on compact sets.
By Lemma~\ref{lem:one-lime}
$\hat{\delta} = H(\delta)$ whenever $\delta \neq \text{NaO}$.
Thus, if $W$ is a neighborhood of $\hat{\delta}$, then $H_n$ maps $B$ into
$W$ for sufficiently large $n$, which implies $\hat{\delta}_n \in W$ for
sufficiently large $n$, which implies $\hat{\delta}_n \to \hat{\delta}$.
\end{proof}

\begin{theorem} \label{th:newt}
Theorem~\ref{th:main} holds
with $G$ replaced by $G^\infty$.
\end{theorem}
Just change the proof to use
Lemmas~\ref{lem:infty-lime} and~\ref{lem:infty-lemon}
instead of
Lemmas~\ref{lem:one-lime} and~\ref{lem:one-lemon}.

When Newton's method converges, under these conditions
(the Hessian is continuous but not necessarily Lipschitz),
then its convergence is superlinear \citep[Theorem~6.2.3]{fletcher}.
If the Hessian does happen to be Lipschitz,
then Newton's method converges quadratically \citep[Theorem~6.2.1]{fletcher}.

It is clear from the proof of Lemma~\ref{lem:infty-lemon} that the argument
of Section~\ref{sec:bound} carries over to this situation.  Now we evaluate
$\nabla q_n$ and $\nabla^2 q_n$ at an infinite sequence of points, but as
the last paragraph of the proof makes clear, all but the first of these
points lie in an arbitrarily small neighborhood of $\hat{\delta}_n$
for sufficiently large $n$.  There is ``almost'' no difference between
one-step Newton and Newton iterated to convergence.

Please note that this appendix is an endorsement of Newton's method
only for objective functions that are ``nearly'' strictly concave
quadratic so that Newton's method ``nearly'' converges in one step.
If Newton's method does not ``nearly'' converge in one step,
then one is crazy to use it and should instead use some form of
``safeguarding'' such as line search \citep[pp.~21--40]{fletcher}
or trust regions \citep[Chapter~5]{fletcher}.

\section{Comparison with Classical Theorems} \label{app:usual}

This appendix compares our ``regularity
conditions'' with the ``usual regularity conditions'' for maximum likelihood.
So we leave ``no $n$'' territory and return to
the conventional ``$n$ goes to infinity'' story.
We adopt ``usual regularity conditions'' similar to those of
\cite[Chapter~18]{ferguson}.  Suppose we have
independent and identically distributed data $X_1$, $X_2$, $\ldots$, $X_n$,
so the log likelihood is the sum of
independent and identically distributed terms
$$
   l_n(\theta) = \sum_{i = 1}^n h_i(\theta)
$$
where
$$
   h_i(\theta) = \log \left( \frac{f_\theta(X_i)}{f_\psi(X_i)} \right)
$$
and where $\psi$ is the true parameter value.  We assume
\begin{enumerate}
\item[(a)] the parameter space is all of $\real^p$,
\item[(b)] $h_i$ is twice continuously differentiable,
\item[(c)] There exists an integrable random variable $M$ such that
$$
   \bignorm{ \nabla^2 h_i(\theta) } \le M,
   \qquad \text{for all $\theta$ in some neighborhood of $\psi$},
$$
(the norm here being the sup norm),
\item[(d)] the expected Fisher information matrix
$$
   K = - E_\psi\{\nabla^2 h_i(\psi)\}
$$
is positive definite, and
\item[(e)] the identity $\int f_\theta(x) \lambda(d x) = 1$, can be
differentiated under the integral sign twice.
\end{enumerate}
\begin{lemma} \label{lem:kumquat}
Under assumptions (a) through (e) above, \eqref{eq:assume-key}
holds with $\tau_n = \sqrt{n}$.
\end{lemma}
\begin{proof}
Assumption (e) implies
$$
   \var_\psi\{\nabla h_i(\psi)\} = - E_\psi\{\nabla^2 h_i(\psi)\}
$$
by differentiation under the integral sign
\citep[p.~120]{ferguson}, assumption (d) implies both sides are
equal to the positive definite matrix $K$, and
the central limit theorem (CLT) and assumption (e) imply
% \begin{subequations}
\begin{equation} \label{eq:clt}
   \nabla q_n(0) = \frac{1}{\sqrt{n}} \sum_{i = 1}^n \nabla h_i(\psi)
   \weakto \NormalDis(0, K).
\end{equation}

Define
$$
   r_n(\delta) = \delta' \nabla q_n(0)
   - \frac{1}{2} \delta' K \delta
$$
and its derivatives
\begin{align*}
   \nabla r_n(\delta) & = \nabla q_n(0) - K \delta
   \\
   \nabla^2 r_n(\delta) & = - K
\end{align*}
\begin{subequations}
We now show that
\begin{equation} \label{eq:interesting}
   r_n \weakto q,
   \qquad \text{in $C^2(\real^p)$}
\end{equation}
is a consequence of the
mapping theorem \citep[Theorem~2.7]{billingsley} and \eqref{eq:clt}.
Define a function $F : \real^p \to C^2(\real^p)$ by
$$
   F(z)(\delta) = z' \delta - \tfrac{1}{2} \delta' K \delta
$$
We claim that $F$ is continuous, which says no more than that
$$
   z_n \to z \qquad \text{and} \qquad \delta_n \to \delta
$$
imply
\begin{align*}
   z_n' \delta_n - \tfrac{1}{2} \delta_n' K \delta_n
   & \to
   z' \delta - \tfrac{1}{2} \delta' K \delta
   \\
   z_n - K \delta_n
   & \to
   z - K \delta
\end{align*}
which are obvious (and that second derivatives converge, which is trivial).
Then an application of the mapping theorem in conjunction with
\eqref{eq:clt} says
\begin{equation} \label{eq:map}
   F\bigl( \nabla q_n(0) \bigr)
   \weakto
   F( Y )
\end{equation}
where $Y$ is a random vector that has the distribution on the right hand
side of \eqref{eq:clt}, and that is the desired conclusion:
\eqref{eq:interesting} in different notation.
\end{subequations}

Now we use Slutsky's theorem
\citep[Theorem~3.1]{billingsley}.  This says that if we can show
$$
   q_n - r_n \probto 0,
   \qquad \text{in $C^2(\real^p)$},
$$
then we are done.  Using the isomorphism of $C^2(\real^p)$ to
a subspace of $C(\real^p)^{1 + p + p \times p}$ and another application
of Slutsky, it is enough to show the separate convergences
\begin{subequations}
\begin{align}
   q_n - r_n
   & \probto 0
   \label{eq:slut:a}
   \\
   \nabla q_n - \nabla r_n
   & \probto 0
   \label{eq:slut:b}
   \\
   \nabla^2 q_n - \nabla^2 r_n
   & \probto 0
   \label{eq:slut:c}
\end{align}
\end{subequations}
which take place
in $C(\real^p)$,
in $C(\real^p)^p$, and
in $C(\real^p)^{p \times p}$, respectively.
But these are equivalent to
\begin{subequations}
\begin{align}
   \sup_{\delta \in B} \abs{q_n(\delta) - r_n(\delta)}
   & \probto 0
   \label{eq:slut:ordinaire:a}
   \\
   \sup_{\delta \in B} \bignorm{\nabla q_n(\delta) - \nabla r_n(\delta)}
   & \probto 0
   \label{eq:slut:ordinaire:b}
   \\
   \sup_{\delta \in B} \bignorm{\nabla^2 q_n(\delta) - \nabla^2 r_n(\delta)}
   & \probto 0
   \label{eq:slut:ordinaire:c}
\end{align}
\end{subequations}
holding for every compact set $B$, which can be seen using
the metric \eqref{eq:metric-rudin}.

Let $B(\theta, \epsilon)$ denote the closed ball
in $\real^p$ centered at $\theta$ with radius $\epsilon$.  Then assumption
(c) can be stated more precisely as the existence of an $\epsilon > 0$
and an integrable random variable $M$ such that
\begin{equation} \label{eq:assume-c}
   \bignorm{ \nabla^2 h_i(\theta) } \le M,
   \qquad \theta \in B(\psi, \epsilon).
\end{equation}
Define
$$
   H_n(\theta) = - \frac{1}{n} \sum_{i = 1}^n \nabla^2 h_i(\theta)
$$
and
$$
   H(\theta) = - E_\psi\{ \nabla^2 h_i(\theta) \}.
$$
Theorem~16(a) in \citet{ferguson} says that
\begin{equation} \label{eq:fergie}
   \sup_{\theta \in B(\psi, \epsilon)}
   \bignorm{ H_n(\theta) - H(\theta) } \asto 0
\end{equation}
and that $H$ is continuous on $B(\psi, \epsilon)$,
the latter assertion appearing in the proof rather
than in the theorem statement.
Note that $H(\psi) = K$.  Also note that
\begin{align*}
   \nabla^2 q_n(\delta) & = - H_n(\psi + n^{- 1 / 2} \delta)
   \\
   \nabla^2 r_n(\delta) & = - H(\psi)
\end{align*}
Hence
\begin{align*}
   \sup_{\delta \in B(0, \eta)}
   \bignorm{\nabla^2 q_n(\delta) - \nabla^2 r_n(\delta)}
   & =
   \sup_{\delta \in B(0, \eta)}
   \bignorm{ H_n(\psi - n^{- 1 / 2} \delta ) - H(\psi) }
   \\
   & =
   \sup_{\theta \in B(\psi, n^{- 1 / 2} \eta)}
   \bignorm{ H_n(\theta) - H(\psi) }
   \\
   & \le
   \sup_{\theta \in B(\psi, n^{- 1 / 2} \eta)}
   \bignorm{ H_n(\theta) - H(\theta) }
   \\
   & \qquad
   +
   \sup_{\theta \in B(\psi, n^{- 1 / 2} \eta)}
   \bignorm{ H(\theta) - H(\psi) }
\end{align*}
and the first term on the right hand side is dominated by the
left hand side of \eqref{eq:fergie} for $n$ such
that $n^{- 1 / 2} \eta \le \epsilon$ and hence converges in probability
to zero by \eqref{eq:fergie}, and the second term on the right hand side
goes to zero by the continuity of $H$.
Since this argument works for arbitrarily large $\eta$,
it proves \eqref{eq:slut:ordinaire:c}.

Now using some of the facts established above and
the Maclaurin series
$$
   \nabla q_n(\delta)
   =
   \nabla q_n(0)
   +
   \int_0^1 \nabla^2 q_n(s \delta) \delta \, d s
$$
we get
$$
   \nabla q_n(\delta) - \nabla r_n(\delta)
   =
   \int_0^1 \left[ H_n(\psi + n^{- 1 / 2} s \delta) - H(\psi) \right] \delta
   \, d s.
$$
So
\begin{align*}
   \sup_{\delta \in B(0, \eta)}
   \bignorm{\nabla q_n(\delta) - \nabla r_n(\delta)}
   & \le
   \sup_{0 \le s \le 1}
   \sup_{\delta \in B(0, \eta)}
   \bignorm{ H_n(\psi + n^{- 1 / 2} s \delta) - H(\psi) }
   \norm{ \delta }
   \\
   & \le
   \sup_{\delta \in B(0, \eta)}
   \bignorm{ H_n(\psi + n^{- 1 / 2} \delta) - H(\psi) }
   \eta
\end{align*}
and we have already shown that the right hand side converges
in probability to zero for any fixed $\eta$, however large,
and that proves \eqref{eq:slut:ordinaire:b}.

Similarly using the Maclaurin series
$$
   q_n(\delta)
   =
   \delta' \nabla q_n(0)
   +
   \int_0^1 \delta' \nabla^2 q_n(s \delta) \delta (1 - s) \, d s
$$
we get
$$
   q_n(\delta) - r_n(\delta)
   =
   \int_0^1 \delta' \left[ H_n(\psi + n^{- 1 / 2} s \delta) - H(\psi) \right]
   \delta (1 - s) \, d s
$$
and the argument proceeds similarly to the other two cases.
\end{proof}

\begin{lemma} \label{lem:prune}
The assumption \eqref{eq:main} of Theorem~\ref{th:main} can be weakened
to \eqref{eq:assume-key} and $\tilde{\delta}_n$ being tight if the last
assertion of the theorem is weakened by replacing \eqref{eq:main}
with \eqref{eq:assume-key}.
\end{lemma}
\begin{proof}
The random sequences $\tilde{\delta}_n$ and $q_n$
are marginally tight, either by assumption or
by the converse half of Prohorov's theorem, hence jointly tight
\citep[Problem~5.9]{billingsley}.
Hence by the direct half of Prohorov's theorem there exist
jointly convergent subsequences.
For every subsequence $\bigl( q_{n_k}, \tilde{\delta}_{n_k} \bigr)$
there is a convergent subsubsequence \eqref{eq:main} with $n$ replaced
by $n_{k_l}$.  And this implies all the conclusions of the theorem with
$n$ replaced by $n_{k_l}$.
But since the limits are the same for
all subsequences $n_k$, it follows from the subsequence principle
\citep[Theorem~2.6]{billingsley}
that the conclusions hold as is with $n$ rather than $n_{k_l}$.
\end{proof}

The point of Lemma~\ref{lem:prune} is to justify Lemma~\ref{lem:kumquat}
dealing only with $q_n$ instead of both $q_n$ and $\tilde{\delta}_n$.
In the conventional ``$n$ is sample size going to infinity'' story,
Lemma~\ref{lem:prune} seems important, saying that it is enough that
$\tilde{\theta}_n$ be a $\tau_n$-consistent sequence of estimators.
In the ``no $n$'' world, however, Lemma~\ref{lem:prune} loses significance
because $\tau_n$-consistency is meaningless when $n$ is not reified.

\bibliographystyle{imsart-nameyear}

\bibliography{simple}

% \begin{center} \LARGE REVISED DOWN TO HERE \end{center}

\end{document}